               \def\version{18 January, 2011}                   %

\documentclass[reqno,11pt]{amsart}

\usepackage{amsmath, latexsym, amssymb, srcltx, stmaryrd}

\newfam\Bbbfam
\font\tenBbb=msbm10
\font\sevenBbb=msbm7
\font\fiveBbb=msbm5
\textfont\Bbbfam=\tenBbb
\scriptfont\Bbbfam=\sevenBbb
\scriptscriptfont\Bbbfam=\fiveBbb

\newcommand{\R}     {\mathbb{R}}
\newcommand{\Z}     {\mathbb{Z}}
\newcommand{\N}     {\mathbb{N}}
\renewcommand{\P}   {\mathbb{P}}

\newcommand{\E}     {\mathbb{E}}
\newcommand{\ssup}[1] {{\scriptscriptstyle{({#1}})}}

\def\1{{\mathchoice {1\mskip-4mu\mathrm l}      
{1\mskip-4mu\mathrm l}
{1\mskip-4.5mu\mathrm l} {1\mskip-5mu\mathrm l}}}

\def\comment#1{}
\newtheoremstyle{thm}{2ex}{2ex}{\itshape\rmfamily}{}
{\bfseries\rmfamily}{}{1.7ex}{}

\newtheoremstyle{rem}{1.3ex}{1.3ex}{\rmfamily}{}
{\itshape\rmfamily}{}{1.5ex}{}

\newtheorem{theorem}{Theorem}[section]
\newtheorem{lemma}[theorem]{Lemma}

\newtheorem{step}{STEP}
\theoremstyle{definition}

\newtheorem{bem}[theorem] {Remark}
\newcommand{\en}       {\end{equation}}
\newcommand{\eq}       {\begin{equation}}
\newcommand{\eqry}   {\begin{eqnarray}}
\newcommand{\enqry}   {\end{eqnarray}}
\newcommand{\eqarray}   {\begin{eqnarray}}
\newcommand{\enarray}   {\end{eqnarray}}
\newcommand{\eqarraystar} {\begin{eqnarray*}}
\newcommand{\enarraystar} {\end{eqnarray*}}
\newcommand{\bel}{\begin{lemma}}
\newcommand{\el}{\end{lemma}}
\newcommand{\bes}{\begin{step}}
\newcommand{\es}{\end{step}}
\newcommand{\bea}{\begin{array}}
\newcommand{\ea}{\end{array}}
\newcommand{\bpr}{\begin{proof}}
\newcommand{\epr}{\end{proof}}

\newcommand{\norm}{|\!|\!|}

\renewcommand{\section}{\secdef\sct\sect}
\newcommand{\sct}[2][default]{\refstepcounter{section}
\vspace{0.5cm}
\setcounter{equation}{0}
\centerline{ 
\scshape \arabic{section}.\ #1}
\vspace{0.3cm}}
\newcommand{\sect}[1]{
\vspace{0.5cm}
\centerline{\large\scshape #1}
\vspace{0.3cm}}

\renewcommand{\subsection}{\secdef \subsct\sbsect}
\newcommand{\subsct}[2][default]{\refstepcounter{subsection}
\nopagebreak
\vspace{0.5\baselineskip}
{\flushleft\bf \arabic{section}.\arabic{subsection}~\bf #1  }
\nopagebreak}
\newcommand{\sbsect}[1]{\vspace{0.1cm}\noindent
{\bf #1}\vspace{0.1cm}}
\newcommand{\heap}[2]{\genfrac{}{}{0pt}{}{#1}{#2}}

\renewcommand{\subsubsection}{%
\secdef \subsubsect\sbsbsect}
\newcommand{\subsubsect}[2][default]{%
\refstepcounter{subsubsection}
\nopagebreak
\vspace{0.1\baselineskip}
\nopagebreak
{\flushleft
\sffamily\slshape
\arabic{section}.\arabic{subsection}.\arabic{subsubsection}
\ %
\sffamily #1\/.}\ }
\newcommand{\sbsbsect}[1]{\vspace{0.1cm}\noindent
{\bf #1}\ }

\renewcommand{\d}{{\rm d}}
\newcommand{\e}{{\rm e}}

\newcommand{\eps}{\varepsilon}
\newcommand{\Sym}{\mathfrak{S}}

\newcommand{\supp}{{\operatorname {supp}}}

\newcommand{\smfrac}[2]{\textstyle{\frac{#1}{#2}}}

\newcommand{\Ccal}   {{\mathcal C }}

\newcommand{\Fcal}   {{\mathcal F }}

\newcommand{\Ical}   {{\mathcal I }}

\newcommand{\Mcal}   {{\mathcal M }}

\begin{document}

\title[Exponential moments of self-intersection local times]{\large 
Self-intersection local times of random walks:\\ 
Exponential moments in subcritical dimensions}
\author[Mathias Becker and Wolfgang K\"onig]{}
\maketitle
\thispagestyle{empty}
\vspace{-0.5cm}

\centerline{\sc By Mathias Becker$ $\footnote{Weierstrass Institute for Applied Analysis and Stochastics,
Mohrenstr. 39, 10117 Berlin, Germany, {\tt becker@wias-berlin.de} and {\tt koenig@wias-berlin.de}} and Wolfgang K\"onig$ $\footnotemark[1]$^,$\footnote{Technical University Berlin, Str. des 17. Juni 136,
10623 Berlin}}
\vspace{0.4cm}
\centerline{\textit{WIAS Berlin, and WIAS Berlin and TU Berlin}}
\vspace{0.8cm}
\centerline{\textit{\version}}
\vspace{0.8cm}

\begin{quote}{\small {\bf Abstract:} Fix $p>1$, not necessarily integer, with $p(d-2)<d$. We study the $p$-fold self-intersection local time of a simple random walk on the lattice $\Z^d$ up to time $t$. This is the $p$-norm of the vector of the walker's local times, $\ell_t$. We derive precise logarithmic asymptotics of the expectation of $\exp\{\theta_t \|\ell_t\|_p\}$ for scales $\theta_t>0$ that are bounded from above, possibly tending to zero. The speed is identified in terms of mixed powers of $t$ and $\theta_t$, and the precise rate is characterized in terms of a variational formula, which is in close connection to the {\it Gagliardo-Nirenberg inequality}. As a corollary, we obtain a large-deviation principle for $\|\ell_t\|_p/(t r_t)$ for deviation functions $r_t$ satisfying $t r_t\gg\E[\|\ell_t\|_p]$. 
Informally, it turns out that the random walk homogeneously squeezes in a $t$-dependent box with diameter of order $\ll t^{1/d}$ to produce the required amount of self-intersections. Our main tool is an upper bound for the joint density of the local times of the walk.}

\end{quote}

\vfill




\setcounter{section}{0}
\section{Introduction}\label{Intro}

\noindent In this paper, we give precise logarithmic asymptotics for the exponential moments of self-intersection local times of random walks on $\Z^d$ on various scales. This topic has been studied a lot in the last decade, since it is a natural question, and a rich phenemonology of critical behaviours of the random walk arises, depending on the dimension, the intersection parameter, the scale, and the type of the random process. Furthermore, the question is technically very difficult to handle, due to bad continuity and boundedness properties of the self-intersection local time. A couple of different techniques for studying self-intersections have been introduced yet, which turned out to be more or less fruitful in various situations. 

Self-intersections of random paths are not only particularly fundamental objects in the theory of stochastic processes, but play also a role in Euclidean quantum field theory \cite{V69}, in the study of random polymer models in Physics \cite{F60} and in Chemistry \cite{dG79}. Furthermore, they appear in the study of models of random paths in random media, which are often motivated by Physics. Let us mention as examples the parabolic Anderson model \cite{GK05} (the heat equation with random potential, describing random mass transport through random media), the closely related problem of random walks in random scenery \cite{AC07} (describing random transport through porous media \cite{MM80}) and the Wiener-sausage problem \cite{DV75}, \cite{DV79}. While a random polymer has a self-repellent path interaction (which we do not consider here), a random medium typically imposes a self-attractive interaction, which is the regime under interest here.

In this paper, we introduce a recently developed method to the study of self-intersections, which enables us to derive limits in terms of an explicit variational formula describing the asymptotics; this formula explains the optimal behaviour of the random walk to produce many self-intersections.  We are working in sub-critical dimensions, where this behaviour consists of a homogeneous squeezing of the path over the whole time interval in a box of a certain time-dependent diameter.

Our method is strongly influenced by the celebrated Donsker-Varadhan large-deviations theory. The main obstacle that has to be overcome to make these ideas work is the lack of continuity, and this is serious. To overcome this, we use an explicit upper bound for the joint density of the walker's local times, which has been derived by Brydges, van der Hofstad and K\"onig \cite{BHK07}. The main task left after applying this bound is to identify the scaling limit of the arising formula, and this is the main novelty of the present paper.

\subsection{Self-intersection local time}\label{sec-SILT}

\noindent Let $(S_t)_{t\in[0,\infty)}$ be a simple random walk in continuous time in $\Z^d$ started from the origin.
We denote by $\P$ the underlying probability measure and by $\E$ 
the corresponding expectation. The main object of this paper is the {\em self-intersection local time} of the random walk. 
In order to introduce this object, we need the local times of the 
random walk at time $t>0$,
\begin{equation}\label{loctim}
\ell_t(z)=\int_{0}^{t}\1_{\{S_r=z\}}\,\d r,\qquad \mbox{ for } z\in\Z^d.
\end{equation}
Fix $p\in(1,\infty)$ and consider the $p$-norm of the local times:
\begin{equation}\label{Lambdadef}
\|\ell_t\|_p
=\Big(\sum_{z\in\Z^d}\ell_t(z)^p\Big)^{1/p},\qquad\mbox { for } t>0.
\end{equation}
If $p$ is an integer, then, clearly
\begin{equation}\label{selfinter}
\|\ell_t\|_p^p=\int_0^t\d t_1\dots\int_0^t\d t_p\,\1_{\{S_{t_1}=\dots=S_{t_p}\}}
\end{equation}
is equal to the $p$-fold self-intersection local time 
of the walk, i.e., the amount of time tuples that it spends in $p$-fold self-intersection sites. For $p=2$, this is usually called the {\it self-intersection local time}. For $p=1$, $\|\ell_t\|_p^p$ is 
just the number $t$, and for $p=0$, it is equal to $\#\{S_r\colon r\in[0,t]\}$, the 
{\it range\/} of the walk. It is certainly also of interest to study $\|\ell_t\|_p^p$ for non-integer
values of $p>1$, see for example~\cite{HKM04}, where this received technical importance. 

The typical behaviour of $\|\ell_t\|_p^p$ for continuous-time random walks cannot be found in the literature, to the best of our knowledge, but we have no doubt that it is, up to the value of the prefactor, equal to the behaviour of the self-intersection local time, $\|\ell_n\|_p^p$, of a centred random walk in discrete time. This has been identified as
\begin{equation}\label{typical}
\E[\|\ell_n\|_p^p]\sim  C a_{d,p}(n),\qquad \mbox{where}\qquad a_{d,p}(n)=
\begin{cases}
n^{(p+1)/2}&\mbox{if }d=1,\\
n (\log n)^{p-1}&\mbox{if }d=2,\\
n&\mbox{if }d\geq 3,
\end{cases}
\end{equation}
where
\begin{equation}\label{typicalconst}
C=\begin{cases}
\frac{\Gamma(p+1)}{\left(2\pi \sqrt{\det\Sigma}\right)^{p-1}}&\mbox{if }d= 2,\\
\gamma^2\sum_{j=1}^\infty j^p(1-\gamma)^{1-j}&\mbox{if } d\geq 3,
\end{cases}
\end{equation}
where $\gamma=\P(S_n\not=0\mbox{ for any }n\in\N)$ denotes the escape probability and $\Sigma$ the covariance matrix of the random walk. See \cite{C07} for $d=2$ and \cite{BK07} for $d\geq 3$, but we could find no reference for $d=1$.

\subsection{Main results}\label{sec-results}

\noindent In this paper, we study the behaviour of the random walk when the walker produces extremely many self-intersections. We restrict to the \emph{subcritical} dimensions, where $d(p-1)<2p$. Before we formulate our results, let us informally describe the optimal behaviour to produce many self-intersections in these dimensions. It is a homogeneous self-squeezing strategy: the walker does not leave a box with radius on a particular scale $\alpha_t\ll \sqrt t$ (we write $b_t\ll c_t$ if $\lim_{t\to\infty} c_t/b_t=\infty$), and the sizes of all the local times are on the same scale $t/\alpha_t^d$ within this box. Furthermore, their rescaled shape approximates a certain deterministic profile, which is given in terms of a characteristic variational formula.

In our result, we do not prove this path picture, but we derive precise logarithmic asymptotics, as $t\to\infty$, for the exponential moments of $\theta_t\|\ell_t\|_p$ for various choices of weight functions $\theta_t\in(0,\infty)$ that are bounded from above. As a direct consequence, this leads to asymptotics of the probability of the event $\{\|\ell_t\|_p>tr_t\}$ for various choices of scale functions $r_t\in[0,1]$ satisfying $a_{d,p}(t)\ll tr_t$.

In order to formulate the result, we have to introduce more notation. By $L^q(\R^d)$ we denote the usual Lebesgue space, which is equipped with the norm $\norm\cdot\norm_q$ if $q\geq 1$. (Since the counterplay between functions defined on $\Z^d$ and on $\R^d$ will be crucial in this paper, we decided to distinguish the norm symbols.) The space $H^1(\R^d)$ is the usual Sobolev space. By $\Mcal_1(X)$ we denote the set of probability measures on a metric space $X$ equipped with the Borel sigma field. We regard $\Mcal_1(\Z^d)$ as a subset of $\ell^1(\Z^d)$.

Now we formulate our main result. 

\begin{theorem}[Exponential Moments]\label{uptails2}
Assume that $d\in\N$ and $p>1$.
\begin{itemize}
\item[(i)]For any $\theta>0$,
\begin{equation}\label{DiscreteExpMom}
\lim_{t\rightarrow \infty}\frac{1}{t}\log\E\Big( \e^{\theta\| \ell_t\|_p }\Big)=\rho_{d,p}^{\ssup{\d}}(\theta),
\end{equation}
where
\begin{equation}\label{rhoddef}
\rho_{d,p}^{\ssup{\d}}(\theta)=\sup\left\{\theta\| \mu\|_p-J(\mu)\colon \mu\in \Mcal_1(\Z^d) \right\}\in(0,\theta],
\end{equation}
and $J(\mu)=\frac {1}{2}\sum_{x\sim y}\big(\sqrt{\mu(x)}-\sqrt{\mu(y)}\big)^2$ denotes the Donsker-Varadhan rate functional.

\item[(ii)] Additionally assume that $d(p-1)< 2p$. Define $\lambda=\frac{2p+d-dp}{2p}\in(0,1)$ and let $\left(\theta_t\right)_{t>0}$ be a function in $(0,\infty)$ such that
\begin{equation}\label{speed}
\left(\frac{\log t}{t}\right)^{\frac {2\lambda}{d+2}}\ll\theta_t\ll 1.
\end{equation}
Furthermore, put
\begin{equation}\label{rhocdef}
\rho_{p,d}^{\ssup{\rm c}}(\theta) = \sup\Big\{\theta |\!|\!| g^2|\!|\!|_p-\frac{1}{2}\norm \nabla g\norm_2^2\colon g\in H^1(\R^d), \norm g\norm _2=1\Big\},\qquad\theta>0.
\end{equation}
Then $\rho_{p,d}^{\ssup{\rm c}}(\theta)\in (0,\infty)$, and
\begin{itemize}
\item[(a)]
\begin{equation}\label{ExpMom1}
\frac{1}{t}\log\E\Big( \e^{\theta_t\| \ell_t\|_p}\Big)\geq\theta_t^{1/\lambda}\big(\rho_{d,p}^{\ssup{\rm c}}(1)+o(1)\big),\qquad t\to\infty.
\end{equation}

\item[(b)]If, additionally to $d(p-1)< 2p$, the stronger assumption $d(p-1)< 2$ is fulfilled, then
\begin{equation}\label{ExpMom2}
\frac{1}{t}\log\E\left( e^{\theta_t\| \ell_t\|_p }\right)\leq\theta_t^{1/\lambda}\big(\rho_{d,p}^{\ssup{\rm c}}(1)+o(1)\big),\qquad t\to\infty.
\end{equation}
\end{itemize}

\end{itemize}
\end{theorem}

Note that we derive the logarithmic asymptotics for exponential moments with fixed parameter $\theta$ in any dimension, but with vanishing parameter $\theta_t$ only in subcritical dimensions $d<\frac{2p}{p-1}$.

A few months after the appearance of the first version of this paper, C.~Laurent \cite{L10b} proved an extension of Theorem~\ref{uptails2}(ii). Indeed, he replaced our assumption $d(p-1)< 2$ by the optimal assumption $d(p-1)< 2p$ and considered all choices of $t^{\frac1p-1}\ll r_t\ll1$. 

A heuristic derivation of Theorem~\ref{uptails2} is given in Section~\ref{outline}. The proof is given in Section~\ref{sec-proof}. Some comments on the related literature are given in Section~\ref{sec-Lit}. We proceed with a couple of remarks.

\subsection{Remarks}\label{sec-remarks}

\begin{bem}[Connection between $\rho_{d,p}^{\ssup{\rm c}}$ and 
$\rho_{d,p}^{\ssup \d}(\theta)$]
Note that the Donsker-Varadhan functional is equal to the walk's Dirichlet form at $\sqrt\mu$, i.e., 
$$
J(\mu)=\frac12\|\nabla^{\ssup{\rm d}}\sqrt\mu\|_2^2,\qquad \mu\in\Mcal_1(\Z^d),
$$
where $\nabla^{\ssup{\rm d}}$ denotes the discrete gradient. Hence, we see that $\rho_{p,d}^{\ssup{\rm c}}(\theta)$ is the continuous version of $\rho_{d,p}^{\ssup{\d}}(\theta)$. 
An important step in our proof of Theorem~\ref{uptails2}(ii) is to show that the continuous version of this formula describes the small-$\theta$ asymptotics of the discrete one, i.e.,
\begin{equation}\label{mainstep}
\rho_{d,p}^{\ssup{\d}}(\theta)\sim \theta^{1/\lambda}\rho_{d,p}^{\ssup{\rm c}}(1),\qquad \theta\downarrow 0.
\end{equation}
(Actually, we only prove a version of this statement on large boxes, see Lemma~\ref{lem-main}.)
In the light of this, we can heuristically explain the transition between the two cases in Theorem \ref{uptails2}. Indeed, if we use (i) for $\theta$ replaced by $\theta_t\rightarrow 0$, then we formally obtain
\begin{equation}
\frac 1 t\log\E\left(\e^{\theta_t\|\ell_t\|_p}\right) \sim \rho^{\ssup{\rm  d}}_{p,d}(\theta_t)\sim \theta_t^{1/\lambda}\rho_{d,p}^{\ssup{\rm  c}}(1).
\end{equation}
Hence, \eqref{mainstep} shows that Theorem~\ref{uptails2}(i) and (ii) are consistent.
\hfill$\Diamond$
\end{bem}

\begin{bem}[On the constant $\rho_{d,p}^{\ssup{\rm c}}(\theta)$]
\label{RhoConnection}
We will show in the following that
\begin{equation}\label{rhoident}
\rho_{p,d}^{\ssup{\rm c}}(\theta)=\theta^{1/\lambda}\lambda\left(\frac{2p}{d(p-1)}\chi_{d,p}\right)^{\frac{\lambda-1}{\lambda}},\qquad \theta\in(0,\infty),
\end{equation}
where
\begin{equation}\label{chiident}
\chi_{d,p}=
\inf\Bigl\{\frac 1{2}\norm \nabla g\norm _2^2\colon g\in L^{2p}(\R^d)\cap L^2(\R^d)\cap H^1(\R^d)
\colon\norm g\norm _2=1=\norm g\norm _{2p}\Bigr\}.
\end{equation}
It turned out in \cite[Lemma~2.1]{GKS04} that $\chi_{d,p}$ is positive if and only 
if $d(p-1)\leq 2p$, i.e., in particular in the cases considered in the present paper. This implies in particular, that $\rho_{p,d}^{\ssup{\rm c}}(\theta)$ is finite and positive for any $\theta>0$. Because of \eqref{mainstep}, also $\rho_{p,d}^{\ssup{\rm d}}(\theta)$ is finite and positive, for any sufficiently small $\theta\in(0,\infty)$. By monotonicity, it is positive for any $\theta\in(0,\infty)$. It is also finite (even not larger than $\theta$), since $J(\mu)\geq0$ and $\|\mu\|_p\leq 1$ for any $\mu\in\Mcal_1(\Z^d)$.

Let us now prove \eqref{rhoident}. In the definition \eqref{rhocdef} of $\rho_{p,d}^{\ssup{\rm c}}(\theta)$, we replace $g$, for any $\beta\in(0,\infty)$, with  $\beta^{d/2}g(\beta \,\cdot)$, which is also $L^2$-normalized. This gives, for any $\beta>0$,
$$
\rho_{p,d}^{\ssup{\rm c}}(\theta)=\sup_{g\in H^1(\R^d)\colon \norm g\norm _2=1}\Big\{ \theta\beta^{\frac{d(p-1)}{p}}\norm g^2\norm _p-\frac{1}{2}\beta^2\norm \nabla g\norm _2^2\Big\}.
$$
Picking the optimal value
$$
\beta^*=\Big(\theta\frac{d(p-1)}p\frac{\norm  g^2\norm _p }{\norm  \nabla g\norm ^2_2}\Big)^{1/(2\lambda)},
$$ 
we get
$$
\rho_{p,d}^{\ssup{\rm c}}(\theta)=\theta^{1/\lambda}\lambda\sup_{g\in H^1(\R^d)\colon \norm g\norm _2=1}\Big(\frac{2p}{d(p-1)}\Big[\frac{\frac 12\norm  \nabla g\norm _2^2 }{\norm  g\norm _{2p}^{2/(1-\lambda)} }\Big]\Big)^\frac{\lambda-1}{\lambda}.
$$
Note that the term in square brackets remains invariant under the transformation $g\mapsto\beta^{d/2}g(\beta\,\cdot )$, which keeps the $L^2$-norm fixed. Thus we may freely add the condition that $\norm  g\norm _{2p}=1$. Recall \eqref{chiident} to see that the proof of \eqref{rhoident} is finished.
\hfill$\Diamond$
\end{bem}

\begin{bem}[Relation to the Gagliardo-Nirenberg constant]\label{GagNirrel}
In dimensions $d\geq 2$, the constant $\chi_{d,p}$ in \eqref{chiident} can be identified in terms 
of the {\it Gagliardo-Nirenberg\/} constant, $K_{d,p}$, as follows. Assume that 
$d\geq 2$ and $p<\frac d{d-2}$. Then $K_{d,p}$ is defined as the smallest constant 
$C>0$ in the {\it Gagliardo-Nirenberg inequality\/} 
\begin{equation}\label{GagNir}
\norm \psi\norm _{2p}\leq C \norm \nabla\psi\norm _2^{\frac{d(p-1)}{2p}}\norm \psi\norm _2^{1-\frac{d(p-1)}
{2p}},\qquad\mbox { for }  \psi\in H^1(\R^d).
\end{equation}
This inequality received a lot of interest from physicists and analysts, and it has 
deep connections to Nash's inequality and logarithmic Sobolev inequalities.
Furthermore, it also plays an important role in the work of Chen \cite{Ch03}, 
\cite{BC04} on intersections of random walks and self-intersections of Brownian 
motion. See \cite[Sect.~2]{Ch03} for more on the Gagliardo-Nirenberg inequality.

It is clear that
\begin{equation}\label{Kdef}
K_{d,p}=\sup_{\heap{\psi\in H^1(\R^d)}{\psi\not=0}}
\frac{\norm \psi\norm _{2p}}{\norm \nabla\psi\norm _2^{\frac{d(p-1)}{2p}}
\norm \psi\norm _2^{1-\frac{d(p-1)}{2p}}}
=\Bigl(\inf_{\heap{\psi\in H^1(\R^d)}{\norm \psi\norm _2=1}}
\norm \psi\norm _{2p}^{-\frac {4q}d}\norm \nabla\psi\norm _2^2\Bigr)^{-\frac d{4q}}.
\end{equation}
Clearly, the term over which the infimum is taken remains unchanged if $\psi$ is 
replaced by $\psi_\beta(\cdot)=\beta^{\frac d2}\psi(\cdot\,\beta)$ for any 
$\beta>0$. Hence, we can freely add the condition $\norm \psi\norm _{2p}=1$ and obtain 
that $$K_{d,p}=\chi_{d,p}^{-\frac d{4q}}.$$ In particular, the variational formulas 
for $K_{d,p}$ in \eqref{Kdef} and for $\chi_{d,p}$ in \eqref{chiident} have the 
same maximizer(s) respectively minimizer(s). It is known that \eqref{Kdef} 
has a maximizer, and this is a smooth, positive and rotationally 
symmetric function (see \cite{We83}). Some uniqueness results are in \cite{MS81}.
\hfill$\Diamond$
\end{bem}

\begin{bem}[Large deviations] \label{RemLDP}
In the spirit of the G\"artner-Ellis theorem (see \cite[Sect.~4.5]{DZ98}), from Theorem~\ref{uptails2}(ii) large-deviation principles for $\|\ell_t\|_p$ on various scales follow. Indeed, fix some function $(r_t)_{t>0}$ satisfying
\begin{equation}\label{rtcondition}
\left(\frac {\log t}t\right)^{\frac{d(p-1)}{p(d+2)}}\ll r_t\ll 1 \qquad\textrm{as}\qquad t\to\infty.
\end{equation}
Then, as $t\to\infty$, under the conditions $p>1$ and $d(p-1)< 2$, we have 
\begin{equation}\label{LDP}
\frac 1t\log\P\big(\big\|\ell_t/t\big\|_p\geq r_t\big)\sim-\chi_{d,p} r_t^{\frac {2p}{d(p-1)}}.
\end{equation}
Applying this to $u r_t$ instead of $r_t$, one obtains that $\|\ell_t\|_p/(t r_t)$ satisfies a large-deviation principle on the scale $tr_t^{2p/(d(p-1))}$ with strictly convex and continuous rate function $(0,\infty)\ni u\mapsto \chi_{d,p}\frac{d(p-1)}{2p} u^{2p/(d(p-1))}$.

In order to prove the upper bound in \eqref{LDP}, put $\theta_t=(r_t\lambda/\rho_{d,p}^{\ssup {\rm c}}(1))^{\lambda/(1-\lambda)}$ and note that the assumption in \eqref{speed} is satisfied. Now use the exponential Chebyshev inequality to see that
$$
\P\big(\big\|\ell_t/t\big\|_p\geq r_t\big)
\leq \E\left(\e^{\theta_t \|\ell_t\|_p}\right)\e^{-tr_t\theta_t}.
$$
Finally use \eqref{ExpMom2} and summarize to see that the upper bound in \eqref{LDP} is true. The lower bound is derived in a standard way using an exponential change of measure, like in the proof of the G\"artner-Ellis theorem. The point is that the limiting logarithmic moment generating function of $\theta_t\|\ell_t\|_p$ is differentiable throughout $(0,\infty)$, as is seen from \eqref{ExpMom1} and \eqref{ExpMom2}.

However, it is not clear to us from Theorem~\ref{uptails2}(i) whether or not $\|\ell_t\|_p/t$ satisfies a large-deviation principle. Indeed, it is unclear if the map $\rho_{d,p}^{\ssup \d}$ is differentiable since the map $\mu\mapsto \theta \|\mu\|_p-J(\mu)$ is a difference of convex functions and therefore not necessarily strictly convex. As a result, we do not know if the maximiser is uniquely attained.
\hfill$\Diamond$
\end{bem}

\begin{bem} One might wonder what \eqref{LDP} might look like in the  critical case $p=\frac d{d-2}$. Note that the right-hand side is then equal to $-\chi_{d,d/(d-2)} r_t$, which is nontrivial according to \cite[Lemma~2.1]{GKS04}, recall Remark~\ref{RhoConnection}. However, in $d\geq 3$, \cite[Theorem 2]{C10} shows that \eqref{LDP} holds, for any $t^{\frac 1p-1}\ll r_t\ll 1$, with $\chi_{d,d/(d-2)}$ replaced by $\sup\{\|\nabla^{\ssup{\rm d}} f\|_2^2\colon f\in\ell^{2p}(\Z^d),\|f\|_{2p}=1\}$, which is a discrete version of $\chi_{d,d/(d-2)}$. This interestingly shows that the critical dimension $d=\frac{2p}{p-1}$ seems to exhibit a different regime and is not the boundary regime of the cases considered here.
\hfill$\Diamond$
\end{bem}

\begin{bem}[Restrictions for $r_t$]\label{SpeedGap} The restriction in \eqref{rtcondition} in our Theorem~\ref{uptails2}(ii) is technical and comes from the error terms in \cite[Theorem 2.1, Prop.~3.6]{BHK07}, which is an important ingredient of our proof of the upper bound, see \eqref{errorterms}. Our proof of the lower bound in \eqref{ExpMom1} does not use this and is indeed true in greater generality.
\hfill$\Diamond$
\end{bem}

\begin{bem}[Restriction in the dimension] Unfortunately, our proof method does not work for all the values of $p$ and $d$ in the subcritical dimensions. The point is that in the proof of the statement in \eqref{mainstep}, we have to approximate a certain step function with its interpolating polygon line in $L^{2p}$-sense, and the difference is essentially equal to the gradient of the polygon line. A control in $L^2$-sense is possible by comparison to the energy term, but the required $L^{2p}$-control represents a problem that we did not overcome to full extent, see \eqref{pnorm2norm}.
\hfill$\Diamond$
\end{bem}

\subsection{Literature remarks}\label{sec-Lit}

\noindent For decades, and particularly in this millenium, there is an active interest in the extreme behaviour of self-intersection local times and their connections with the theory of large deviations. The recent monograph \cite{Ch09} comprehensively studies a host of results and concepts on extremely self-attractive paths and the closely related topic of extreme mutual attraction of several independent random paths. This subject is a rich source of various phenomena that arise, depending on the dimension $d$, the intersection parameter $p$ and the scale of the deviation, $r_t$. In particular, an interesting collapse transition in  the path behaviour from subcritical dimensions (which we study here) to supercritical dimensions can be observed: homogeneous squeezing versus short-time clumping.  See \cite{K10} for a concise description of the relevant heuristics and a survey on current proof techniques.

Various methods have been employed in this field, and the monograph \cite{Ch09} gave an inspiration to develop new ideas very recently. Le Gall \cite{Le86} used a bisection technique (introduced by Varadhan in \cite{V69}) of successive division of the path into equally long pieces and controlling the mutual interaction. In the important special cases $d=2=p$ and $d=3, p=2$, Bass, Chen and Rosen \cite[Theorem 1.1]{BCR06} and Chen, respectively, derived our Theorem~\ref{uptails2}(ii) in greater generality, using the bisection method, see also \cite[Theorems 8.2.1 and 8.4.2]{Ch09}. Furthermore, Chen and Li \cite[Theorem~1.3]{CL04} proved it also in $d=1$ for arbitrary $p\in\N$, see \cite[Theorem~8.1.1]{Ch09}. Their strategy made the application of Donsker-Varadhan's large-deviation technique possible by a sophisticated compactification procedure, which uses a lot of abstract functional analysis and goes back to de Acosta. An earlier result by Chen \cite[Theorem~3.1, (3.3)]{Ch03} concerns a smoothed version of Theorem~\ref{uptails2}(ii) for all subcritical values of integers $p$. The bisection method has been further developed by Asselah to other values of $p>1$ in a series of papers, out of which we want to mention \cite{AC07}, \cite{A08}, and \cite{A09}. A combinatorial method that controls the high, $t$-dependent, polynomial moments of the intersection local time was applied in \cite[Prop.~2.1]{HKM04} to derive a weaker statement than Theorem~\ref{uptails2}(ii), still identifying the correct scale. Recently, Castell \cite{C10} used Dynkin's isomorphism for deriving precise logarithmic asymptotics for the deviations of the intersection local times in $d\geq3$ for the critical parameter $p=\frac d{d-2}$, which is the boundary of our restriction $d(p-1)<2p$. See the introduction of \cite {C10} for an extensive but concise summary of related results. Her technique was later adapted to proofs of very-large deviation principles in sub- and in supercritical dimension  \cite{L10a}, \cite{L10b}.

The present paper uses a new strategy that goes back to a formula for the joint density of the local times of any continuous-time finite-state space Markov chain. The kernel is an explicit upper bound for this density, which basically implies the upper bound in Donsker-Varadhan's large-deviation principle for these local times without using any topology. In this way, one obtains a discrete, $t$-dependent variational formula, and the main task is to find its large-$t$ asymptotics. This is done via techniques in the spirit of $\Gamma$-convergence. An example of this technique was carried out in \cite[Section 5]{HKM04}.

\subsection{Heuristic derivation of Theorem~\ref{uptails2}}\label{outline}

We now give a heuristic derivation of Theorem~\ref{uptails2}(ii), which is based on large-deviation theory. 

Recall that $\lambda=(2p-dp+d)/2p\in(0,1)$. For some scale function $\alpha_t\to\infty$, to be specified later, define the random step function $L_t\colon\R^d\to[0,\infty)$ as the scaled normalized version  of the local times $\ell_t$, i.e.,
\begin{equation}\label{Lndef}
L_t(x)=\frac{\alpha_t^d}t \ell_t\bigl(\lfloor x\alpha_t\rfloor\bigr),
\qquad\mbox { for }  x\in \R^d.
\end{equation}
Then $L_t$ is a random element of the set
\begin{equation}
\Fcal=\Bigl\{f\in L^1(\R^d)\colon f\geq 0, \int_{\R^d} f(x)\,\d x=1\Bigr\}
\end{equation}
of all probability densities on $\R^d$. In the spirit of the celebrated large-deviation
theorem of Donsker and Varadhan, if $\alpha_t$ satisfies
$1\ll \alpha_t^d\ll a_{d,0}(t)$ (see \eqref{typical}), 
then the distributions of $L_t$ satisfy a weak large-deviation principle 
in the weak $L^1$-topology on $\Fcal$ with speed $t\alpha_t^{-2}$ and rate 
function $\Ical \colon \Fcal\to[0,\infty]$ given by
\begin{equation}\label{Idef}
\Ical(f)=\begin{cases} \frac 1{2} \bigl\Vert\nabla\sqrt{f} \bigr\Vert_2^2 
&\mbox{if } \sqrt{f}\in H^1(\R^d),\\ \infty&\mbox{otherwise.} 
\end{cases}
\end{equation}
Roughly, this large-deviation principle says that,
\begin{equation}\label{LDPappr}
\P(L_t\in \,\cdot\,)=
\exp\Bigl\{-\frac t{\alpha_t^2} \Bigl[\inf_{f\in\,\cdot}\Ical(f)+o(1)\Bigr]\Bigr\},
\end{equation} 
and the convergence takes place in the weak topology. This principle has been partially proved in a special case in \cite{DV79}, a proof in the general case was given in \cite[Prop.~3.4]{HKM04}.

In order to heuristically recover Theorem~\ref{uptails2}(ii) in terms of
the statement in \eqref{LDPappr}, note that
$$
\theta_t\|\ell_t\|_p=\theta_t\Big(\sum_{z\in\Z^d}\ell_t(z)^p\Big)^{1/p}
= t\theta_t\alpha_t^{-d}\Big(\sum_{z\in\Z^d}L_t\bigl({\textstyle{\frac z{\alpha_t}}}\bigr)^p\Big)^{1/p}
= t\theta_t\alpha_t^{\frac{d(1-p)}{p}}\norm L_t\norm_p.
$$
Now we choose $\alpha_t=\theta_t^{-1/(2\lambda)}$ and have therefore that
\begin{equation}\label{Lambdaident}
\theta_t\|\ell_t\|_p=\frac t{\alpha_t^2}\norm L_t\norm_p\qquad\mbox{and}\qquad t\theta_t^{1/\lambda}=\frac t{\alpha_t^2}.
\end{equation}
Therefore, the scale $t/\alpha_t^2$ of the large-deviation principle coincides with the logarithmic scale $t\theta_t^{1/\lambda}$ of the expectation under interest in Theorem~\ref{uptails2}(ii). A formal application of Varadhan's lemma yields
$$
\E\Big(\e^{\theta_t\|\ell_t\|_p}\Big)=
\E\Big(\exp\Big\{\frac t{\alpha_t^2}\norm L_t\norm _p\Big\}\Big)
=\exp\Big\{\frac t{\alpha_t^2}(\widetilde\rho+o(1))\Big\},
$$
where 
\begin{equation}\label{chitildedef}\begin{aligned}
\widetilde\rho& =\sup\Bigl\{\norm f\norm _p-\Ical(f)\colon f\in\Fcal\Bigr\}\\ 
&  =\sup \big\{\norm g^2\norm _p-\frac 12\norm \nabla g\norm _2^2 \, : \,
g\in L^2(\R^d)\cap L^{2p}(\R^d)\cap H^1(\R^d), \, \norm g\norm _2=1 \big\}\\
&=\rho_{d,p}^{\ssup {\rm c}}(1). 
\end{aligned}
\end{equation}
This ends the heuristic derivation of Theorem~\ref{uptails2}(ii). In the same way, one can derive also Theorem~\ref{uptails2}(i); this is similar to the line of argument used in \cite{GM98}.

Hence, we see informally that, in Theorem~\ref{uptails2}(ii), the main contribution to the exponential moments should come from those random walk realisations that make the rescaled local times, $L_t$, look like the minimiser(s) of the variational formula $\rho_{d,p}^{\ssup {\rm c}}(1)$. In particular, the random walk should stay within a region with diameter $\alpha_t\ll t^{1/d}$, and each local time should be of order $t/\alpha_t^d\gg 1$. That is, there is a time-homogeneous squeezing strategy. In Theorem~\ref{uptails2}(i), the interpretation is analogous, but the diameter of the preferred region is now of finite order in $t$. This is why a discrete picture arises in the variational formula $\rho_{d,p}^{\ssup {\rm d}}(1)$.

There are several serious obstacles to be removed when trying to turn the above 
heuristics into an honest proof: (1) the large-deviation principle only holds on 
compact subsets of $\R^d$, (2) the functional $L_t\mapsto \norm L_t\norm _p$ is not bounded, 
and (3) this functional is not continuous. Removing the 
obstacle (1) is easy and standard, but it is in general notoriously difficult to 
overcome the obstacles (2) and (3) for related problems.

\section{Proof of Theorem~\ref{uptails2}}\label{sec-proof}

\noindent We prove Theorem~\ref{uptails2}(i) (that is, \eqref{DiscreteExpMom}) in Section~\ref{sec-proof1}, the lower-bound part \eqref{ExpMom1} of Theorem~\ref{uptails2}(ii) in Section~\ref{sec-low} and the upper-bound part  \eqref{ExpMom2} in Section~\ref{sec-upp}.

\subsection{Proof of {\bf{(\ref{DiscreteExpMom})}}}\label{sec-proof1}

\noindent This is analogous to the proof of \cite[Theorem~1.2]{GM98}; we will sketch the argument. First we explain the lower bound. Let $Q_R$ denote the box $[-R,R]^d\cap\Z^d$ and insert an indicator on the event $\{\supp(\ell_t)\subset Q_R\}$ in the expectation, to get, for any $\theta>0$,
$$
\E\left(\e^{\theta\|\ell_t\|_p}\right)
\geq \E\left(\e^{\theta t \|\ell_t/t\|_p}\1_{\{\supp(\ell_t)\subset Q_R\}}\right).
$$
Now observe that the functional $\mu\mapsto \|\mu\|_p$ is continuous and bounded on the set $\Mcal_1(Q_R)$ of all probability measures on $Q_R$. Furthermore, under the sub-probability measure $\P(\cdot,\supp(\ell_t)\subset Q_R)$, the distributions of $\ell_t/t$ satisfy a large-deviations principle with scale $t$ and rate function equal to the restriction of $J$ defined in Theorem~\ref{uptails2}(i) to $\Mcal_1(Q_R)$. Hence,
Varadhan's lemma \cite[Lemma~4.3.4]{DZ98} yields that
\begin{equation}\label{prooflowerR}
\liminf_{t\to\infty}\frac 1t\log \E\left(\e^{\theta\|\ell_t\|_p}\right)
\geq \sup_{\mu\in \Mcal_1(Q_R)}\big(\theta \|\mu\|_p-J(\mu)\big).
\end{equation}
Letting $R\to\infty$ and using an elementary approximation argument, we see that the right-hand side converges towards $\rho_{d,p}^{\ssup \d}(\theta)$. This ends the proof of the lower bound in \eqref{DiscreteExpMom}.

Now we explain the upper bound. Introduce the periodized version of the local times in $Q_R$,
\begin{equation}\label{periodloctim}
\ell_t^{\ssup R}(z)=\sum_{x\in\Z^d}\ell_t(z+Rx),\qquad t\in(0,\infty), R\in \N, z\in Q_R.
\end{equation}
Then it is easy to see that $\|\ell_t\|_p\leq \|\ell_t^{\ssup R}\|_{p,R}$, where $ \| \cdot \|_{p,R}$ denotes the $p$-norm for functions $Q_R\to\R$. Hence, for any $\theta\in(0,\infty)$,
$$
\E\big(\e^{\theta \|\ell_t\|_p}\big)\leq \E\big(\e^{\theta \|\ell_t^{\ssup R}\|_{p,R}}\big)
=\E\big(\e^{\theta t \|\ell_t^{\ssup R}/t\|_{p,R}}\big).
$$
It is well-known that $(\ell_t^{\ssup R}/t)_{t>0}$ satisfies a large-deviations principle on the set of probability measures on $Q_R$ with rate function $\mu\mapsto J_{R,{\rm per}}(\mu)$ equal to the Dirichlet form at $\sqrt \mu$ of $-\frac 12\Delta$ in $Q_R$ with periodic boundary condition. By continuity and boundedness of the map $\mu\mapsto \|\mu\|_p$, it is clear from Varadhan's lemma that
\begin{equation}\label{proofupperR}
\limsup_{t\to\infty}\frac 1t\log \E\left(\e^{\theta\|\ell_t\|_p}\right)
\leq \sup_{\mu\in \Mcal_1(Q_R)}\big(\theta \|\mu\|_p-J_{R,{\rm per}}(\mu)\big).
\end{equation}
In the same way as in the proof of \cite[Lemma 1.10]{GM98}, one shows that the difference between the variational formulas on the right-hand sides of \eqref{proofupperR} and \eqref{prooflowerR} vanish in the limit as $R\to\infty$. This ends the proof of \eqref{DiscreteExpMom}.

\subsection{Proof of the lower bound {\bf{(\ref{ExpMom1})}}}\label{sec-low}
 
\noindent Fix $q>1$ with $\frac 1p+\frac  1q=1$, and consider a continuous bounded function $f\colon\R^d\to\R$ such that $\norm f\norm _q=1$. According to H\"older's inequality, we have
$$
\norm L_t\norm _p\geq \langle f,L_t\rangle.
$$
Recall from Section~\ref{outline} that $\alpha_t=\theta_t^{-1/(2\lambda)}$. Using \eqref{Lambdaident}, we obtain, for any $R>0$, the lower bound
\begin{equation}
\E\left(\e^{\theta_t\|\ell_t\|_p}\right)
=\E\left(\e^{t\alpha_t^{-2}\norm L_t\norm _p}\right)
\geq\E\left(\e^{t \alpha_t^{-2}\langle f,L_t\rangle}\1_{\{\supp(L_t)\subset B_R\}}\right),
\end{equation}
where we denote $B_R=[-R,R]^d$. According to \cite[Lemma~3.2]{GKS04}, the distributions of $L_t$ under the sub-probability measure $\P(\cdot\, ,\supp(L_t)\subset B_R)$ satisfy, as $t\to\infty$, a large-deviation principle on the set of probability densities on $\R^d$ with support in $B_R$. The rate function is
$$
g^2\mapsto \frac 12 \norm \nabla g\norm _2^2,\qquad g\in H^1(\R^d), \norm g\norm _2=1,\supp(g)\subset B_R,
$$
and we put the value of the rate function equal to $+\infty$ if $g$ is not in $H^1(\R^d)$ or not normalized or if its support is not contained in $B_R$. The speed is $t\alpha_t^{-2}$, which is identical to the logarithmic scale in \eqref{ExpMom1}, $t\theta_t^{1/\lambda}$. Since the map $L_t\mapsto \langle f,L_t\rangle$ is continuous, we obtain, by Varadhan's lemma, from that large-deviation principle the following estimate for the exponential moments:
$$
\liminf_{t\to\infty}\frac {\alpha_t^2}t\log \E\left(\e^{\theta_t\|\ell_t\|_p}\right)
\geq \sup_{\heap{g\in H^1(\R^d)\colon \norm g\norm _2=1}{\supp(g)\subset B_R}}\left(\langle f,g^2\rangle-\frac 12 \norm \nabla g\norm _2^2\right).
$$
Certainly we can restrict the supremum over $g$ to $g\in L^{2p}(\R^d)$.  Since the left-hand side does not depend on $f$, we can take on the right-hand side the supremum over all continuous bounded $f\colon B_R\to\R$ satisfying $\norm f\norm _q=1$. Using an elementary approximation argument and the duality between $L^q(\R^d)$ and $L^p(\R^d)$, we see that
$$
\sup_{f\in \Ccal(B_R), \norm f\norm _q=1}
\sup_{\heap{g\in H^1(\R^d)\colon \norm g\norm _2=1}{\supp(g)\subset B_R}}\left(\langle f,g^2\rangle-\frac 12 \norm \nabla g\norm _2^2\right)
\geq \sup_{\heap{g\in H^1(\R^d)\cap L^{2p}(\R^d)\colon \norm g\norm _2=1}{\supp(g)\subset B_R}}\left(\norm g^2\norm _p-\frac 12 \norm \nabla g\norm _2^2\right).
$$
Letting $R\to\infty$ and using another elementary approximation argument, we see that the right-hand side converges to $\rho_{d,p}^{\ssup {\rm c}}(1)$. This ends the proof of the lower bound in \eqref{ExpMom1}.

\subsection{Proof of the upper bound {\bf{(\ref{ExpMom2})}}}\label{sec-upp}

\noindent Fix $\theta_t\in(0,\infty)$ satisfying \eqref{speed}. Recall from Section~\ref{outline} that $\lambda=(2p-dp+d)/2p\in(0,1)$ and that $\alpha_t=\theta_t^{-1/(2\lambda)}$. As in the proof of \eqref{DiscreteExpMom}, we estimate from above against a periodized version of the walk, but now in the $t$-dependent box $Q_{R\alpha_t}= [-R\alpha_t,R\alpha_t]^d\cap \Z^d$. Recall that $\ell_t^{\ssup{R\alpha_t}}$ denotes the periodized version of the local times, see \eqref{periodloctim}. We estimate
\begin{equation}\label{BHKbeforeUse}
\E\left(\e^{\theta_t \| \ell_t\|_p}\right)
\leq\E\left(\e^{\theta_t \| \ell_t^{\ssup{R\alpha_t}}\|_{p,R\alpha_t}}\right)
=\E\left(\exp\Big\{t\alpha_t^{-2\lambda}\big\|\smfrac 1t\ell_t^{\ssup{R\alpha_t}}\big\|_{p,R\alpha_t}\Big\}\right).
\end{equation}
Note that $\ell_t^{\ssup{R\alpha_t}}$ is the local time vector of the continuous-time random walk on $Q_{R\alpha_t}$ with generator $A_{R\alpha_t}$, which is $\frac 12$ times the Laplace operator in $Q_{R\alpha_t}$ with periodic boundary condition.

Now we employ a recently developed method for effectively deriving large-deviation upper bounds without continuity and boundedness assumptions. The base of this method has been laid in \cite{BHK07} and has been applied first in \cite[Theorem~3.7]{BHK07}  and \cite[Section 5]{HKM04}. The main point is the identification of a joint density of the local time vector $\ell_t^{\ssup{R\alpha_t}}$ and of an explicit upper bound for this density. In this way, no continuity or boundedness is required, which is a great improvement over classical large-deviations arguments. The upper bound is in terms of a discrete-space variational formula and additional error terms involving the box size. Let us remark that these error terms give us the lower restriction for $\theta_t$ in \eqref{speed}. The main work after the application of the upper bound is to derive the large-$t$ asymptotics of the discrete variational formula, which requires Gamma-convergence techniques.

We apply \cite[Theorem~3.6]{BHK07} to get, for any $t\geq 1$,
\begin{equation}\label{BHKinUse}
\begin{aligned}
\log\E\left(\exp\Big\{t\alpha_t^{-2\lambda}\big\|\smfrac 1t\ell_t^{\ssup{R\alpha_t}}\big\|_{p,R\alpha_t}\Big\}\right)
&\leq  t\sup_{\mu\in \Mcal_1(Q_{R\alpha_t})}\left[\alpha_t^{-2\lambda}\|\mu\|_{p,R\alpha_t}-\| \left(-A_{R\alpha_t}\right)^{1/2}\sqrt{\mu}\|_{2,R\alpha_t}^2 \right]\\
&\qquad+| Q_{R\alpha_t}|\log\left(2d\sqrt{8\e}\,t\right)
+\log | Q_{R\alpha_t}|+ \frac{| Q_{R\alpha_t}|}{4t}
\end{aligned}
\end{equation}
Here we have used that $\eta_{Q_{R\alpha_t}}$, defined in \cite[(3.2)]{BHK07}, is equal to $2d$.

Let us first show that the terms in the second line on the right-hand side are asymptotically negligible on the scale $t/\alpha_t^2=t\theta_t^{1/\lambda}$. Indeed, these terms are of order $\alpha_t^d\log t$, and we see that
\begin{equation}\label{errorterms}
\alpha_t^d\log t=\frac t{\alpha_t^2} \alpha_t^{d+2}\frac {\log t}t=\frac t{\alpha_t^2} 
\theta_t^{-\frac{d+2}{2\lambda}}\frac {\log t}t
\ll\frac t{\alpha_t^2},
\end{equation}
where we used \eqref{speed}.

Hence, substituting this in \eqref{BHKbeforeUse} and \eqref{BHKinUse}, it is clear that, for the proof of the upper bound in \eqref{ExpMom2}, it is sufficient to prove the following.

\begin{lemma}\label{lem-main}
\begin{equation}\label{GammaConv}
\limsup_{R\to\infty}\limsup_{t\to\infty}\alpha_t^2\sup_{\mu\in \Mcal_1(Q_{R\alpha_t})}\left[\alpha_t^{-2\lambda}\|\mu\|_{p,R\alpha_t}-\big\| \left(-A_{R\alpha_t}\right)^{1/2}\sqrt{\mu}\big\|_{2,R\alpha_t}^2 \right]
\leq \rho_{p,d}^{\ssup {\rm c}}(1).
\end{equation}
\end{lemma}

\begin{proof} We will adapt the method described in \cite[Prop.~5.1]{HKM04}. First, we pick sequences $R_n\to\infty$, $t_n\to\infty$ and $\mu_n\in\Mcal_1(Q_{R_n\alpha_{t_n}})$ such that
\begin{equation}\label{Rntnmun}
\mbox{L.h.s.~of \eqref{GammaConv}}
\leq \alpha_{t_n}^{d(p-1)/p}\| \mu_n\|_{p,R_n\alpha_{t_n}}-\alpha_{t_n}^2\big\| \left(-A_{R_n\alpha_{t_n}}\right)^{1/2}\sqrt{\mu_n }\big\|_{2,R_n\alpha_{t_n}}^2 +\smfrac 1n, \qquad n\in\N.
\end{equation}
We may assume that $\mu_n$ is a probability measure on $\Z^d$ with support in $Q_{R_n\alpha_{t_n}}$.

In the following, we will construct a sequence $(h_n)_n$ in $H^1(\R^d)$ such that (1) $h_n$ is $L^2$-normalized, (2)  the term $\alpha_{t_n}^2\| (-A_{R_n\alpha_{t_n}})^{1/2}\sqrt{\mu_n }\|_{2,R_n\alpha_{t_n}}^2$ is approximately equal to its energy, $\frac 12\norm \nabla h_n\norm _2^2$, and (3) the term $\alpha_{t_n}^{d(p-1)/p}\| \mu_n\|_{p,R_n\alpha_{t_n}}$ is approximately equal to $\norm h_n^2\norm _p$. Having constructed such a series, the proof is quickly finished.

We are using finite-element methods to construct such function $h_n$, see \cite{B07} for the general theory. We split $\R^d$ along the integer grid into half-open unit cubes $C(k)=\times _{i=1}^d(k_i,k_i+1]$ with $k=(k_1,\dots, k_d)\in\Z^d$. Each such cube is split into $d!$ \lq tetrahedra\rq\ as follows. For $\sigma\in\Sym_d$, the set of permutations of $1,\dots,d$, we denote by $T_\sigma(k)$ the intersection of $C(k)$ with the convex hull generated by $k,k +\e_{\sigma(1)},\ldots,k +\e_{\sigma(1)}+\ldots+\e_{\sigma(d)}$, where $\e_i$ denotes the $i$-th unit vector in $\R^d$. Up to the boundary, the tetrahedra $T_\sigma(k)$ with $\sigma\in\Sym_d$ are pairwise disjoint. One can easily see that, for $x\in C(k)$,
$$
x\in T_{\sigma}(k)\qquad\Longleftrightarrow \qquad x_{\sigma(1)}-\lfloor x_{\sigma(1)}\rfloor \geq\ldots\geq x_{\sigma(d)}-\lfloor x_{\sigma(d)}\rfloor>0.
$$
The interior of $T_{\sigma}(k)$ is characterised by strict inequalities. A site $x$ belongs to two different of these tetrahedra if and only if at least one of the inequalities is an equality.

Now we introduce the following functions. For $n\in\N$, $i\in\{1,\dots,d\}$ and $y\in\R^d$, put
\begin{equation}\label{fndef}
f_{n,\sigma,i}(y)=\Big[\sqrt{\mu_n\big(\lfloor y \rfloor+{\e}_{\sigma(1)}+\dots+{\e}_{\sigma(i)}\big)}-\sqrt{\mu_n\big(\lfloor y \rfloor+{\e}_{\sigma(1)}+\dots+{\e}_{\sigma(i-1)}\big)}\,\Big]\,(y_{\sigma(i)}-\lfloor y_{\sigma(i)}\rfloor),
\end{equation}
where we use the convention $\sigma(0)=0$ and ${\e}_0=0$.
Furthermore, given $k\in\Z^d$ and $y\in C(k)$, we pick some $\sigma(y)\in\Sym_d$ such that $y\in T_{\sigma(y)}(k)$ and define
\begin{equation}\label{gndef}
g_{n}(x)=\alpha_{t_n}^{d/2}\mu_n(\lfloor \alpha_{t_n}x\rfloor )^{1/2}+\alpha_{t_n}^{d/2}\sum_{i=1}^d f_{n,\sigma(\alpha_{t_n}x),i}(\alpha_{t_n}x).
\end{equation}
Now we show that $g_n$ is well-defined, i.e., that if $y$ lies in $T_{\sigma_1}(k)$ and  in $T_{\sigma_2}(k)$, then 
\begin{equation}
\sum_{i=1}^d f_{n,\sigma_1,i}(y)=\sum_{i=1}^d f_{n,\sigma_2,i}(y).
\end{equation} 
For sake of simplicity, we do this only for the special case
$$
\sigma_1(1)=1 =\sigma_2(2)\qquad  \textrm{and}\qquad \sigma_1(2)=2=\sigma_2(1),\qquad\mbox{and }\sigma_1(i)=\sigma_2(i)\mbox{ for }i\geq 3.
$$
That is, $y_1-\lfloor y_1\rfloor=y_2-\lfloor y_2\rfloor \geq y_i-\lfloor y_i\rfloor$ for any $i\in\{3,\ldots,d\}$. We calculate
$$
\begin{aligned}
\sum_{i=1}^d f_{n,\sigma_1,i}(y)-\sum_{i=1}^d f_{n,\sigma_2,i}(y)
&= \Big[\sqrt{\mu_n\big(\lfloor y \rfloor+\e_1\big)}-\sqrt{\mu_n\big(\lfloor y \rfloor\big)}\Big]    \,(y_1-\lfloor y_1\rfloor)\\
&\quad+\Big[\sqrt{\mu_n\big(\lfloor y \rfloor+\e_1+\e_2\big)}-\sqrt{\mu_n\big(\lfloor y \rfloor+{\e}_1}\Big]    \,(y_2-\lfloor y_2\rfloor)\\
&\quad- \Big[\sqrt{\mu_n\big(\lfloor y \rfloor+{\e}_2\big)}-\sqrt{\mu_n\big(\lfloor y \rfloor}\Big]    \,(y_2-\lfloor y_2\rfloor)\\
&\quad-\Big[\sqrt{\mu_n\big(\lfloor y \rfloor+\e_2+\e_1\big)}-\sqrt{\mu_n\big(\lfloor y \rfloor+{\e}_2}\Big]    \,(y_1-\lfloor y_1\rfloor).
\end{aligned}
$$
This is equal to zero, since $y_1-\lfloor y_1\rfloor=y_2-\lfloor y_2\rfloor$. Hence, we know that $g_n$ is well-defined. Furthermore, this also shows that $g_n$ is continuous within each $C(k)$. From now on, we abbreviate $f_{n,\sigma(y),i}(y)$ by $f_{n,\sigma,i}(y)$.

Similarly, we see that $g_n$ is also continuous at the boundary of each of the cubes. Indeed, a site $y\in\R^d$ belongs to this boundary if and only if it has at least one integer coordinate. For the sake of simplicity assume that only for $i=1$ it holds that $y_i-\lfloor y_i\rfloor =1$. It is clear that $y\in T_{\sigma_1}(k)\cap \overline{T_{\sigma_2}(k+\e_1)}$ where $\sigma_1,\sigma_2\in \Sym_d$ are given by
$$
\sigma_1(1)=1, \sigma_2(d)=1 \textrm{ and } \sigma_1(i+1)=\sigma_2(i)\qquad\forall i\in\{1,\ldots,d-1\}
$$
Choose an arbitrary sequence $(y^{\ssup m})_m\in T_{\sigma_2}(k+\e_1)$ that converges to $y$. For sufficiently large $m$ it holds that $\lfloor y\rfloor+\e_1 =\lfloor y^{\ssup m}\rfloor$. We see now that:
$$
\begin{aligned}
\mu_n&(\lfloor y^{\ssup m}\rfloor)^{1/2}+\sum_{i=1}^d f_{n,\sigma_2,i}(y^{\ssup m})
=\mu_n(\lfloor y\rfloor+\e_1)^{1/2}+\sum_{i=1}^{d}(y^{\ssup m}_{\sigma_2(i)}-\lfloor y^{\ssup m}_{\sigma_2(i)}\rfloor) \\
&\quad\times\Big[\sqrt{\mu_n\big(\lfloor y \rfloor+\e_1+{\e}_{\sigma_2(1)}+\dots+{\e}_{\sigma_2(i)}\big)}-\sqrt{\mu_n\big(\lfloor y \rfloor+\e_1+{\e}_{\sigma_2(1)}+\dots+{\e}_{\sigma_2(i-1)}\big)}\,\Big].
\end{aligned}
$$
Note that the summand for $i=d$ converges to zero, since $\lim_{m\to\infty}(y^{\ssup m}_{\sigma_2(d)}-\lfloor y^{\ssup m}_{\sigma_2(d)}\rfloor)=0$. In the remaining sum on $i=1,\dots,d-1$, we shift the sum by substituting $i=j-1$ and replace $\sigma_2(j-1)$ by $\sigma_1(j)$, to get, as $m\to\infty$,
$$
\begin{aligned}
\mu_n&(\lfloor y^{\ssup m}\rfloor)^{1/2}+\sum_{i=1}^d f_{n,\sigma_2,i}(y^{\ssup m})\\
&=\mu_n(\lfloor y\rfloor )^{1/2}+\Big[\sqrt{\mu_n(\lfloor y\rfloor+\e_1)}-\sqrt{\mu(\lfloor y\rfloor)}\Big](y^{\ssup m}_{\sigma_1(1)}-\lfloor y_{\sigma_1(1)}\rfloor)\\
&\quad+\sum_{j=2}^{d} \Big[\sqrt{\mu_n\big(\lfloor y \rfloor+\e_{\sigma_1(1)}+{\e}_{\sigma_1(2)}+\dots+{\e}_{\sigma_1(j)}\big)}-\sqrt{\mu_n\big(\lfloor y \rfloor+\e_{\sigma_1(1)}+{\e}_{\sigma_1(2)}+\dots+{\e}_{\sigma_1(j-1)}\big)}\Big]\\
&\qquad\qquad\qquad\times(y^{\ssup m}_{\sigma_1(j)}-\lfloor y^{\ssup m}_{\sigma_1(j)}\rfloor)+o(1).
\end{aligned}
$$
As $y^{\ssup m}_{j}\to y_j$, we see that the right-hand side converges towards $\mu_n(\lfloor y\rfloor)^{1/2}+\sum_{j=1}^d f_{n,\sigma_1,j}(y)$. Hence, we proved the continuity of $g_n$ at the border of each $C(k)$ and thus the continuity of $g_n$ on the entire box $B_{R_n}=[-R_n,R_n]^d$. In addition, as $g_n$ is clearly differentiable in the interior of each tetrahedron $T_\sigma(k)$, we get that $g_n$ lies in $H^1(B_{R_n})$. An elementary calculation shows that 
\begin{equation}\label{gnprop1}
\norm \nabla_{\!\!\scriptscriptstyle{R_n}} g_n\norm _{2, R_n}^2=\alpha_{t_n}^2\big\|\left(-A_{R_n\alpha_{t_n}}\right)^{1/2}\sqrt{\mu_n }\big\|_{2,R_n\alpha_{t_n}}^2,\qquad n\in\N,
\end{equation}
where $\norm\cdot\norm_{p,R}$ denotes the $p$-norm of functions $B_R\to \R$, and $\nabla_{\!\!\scriptscriptstyle{R_n}} g_n$ denotes the gradient of $g_n$ in the box $B_{R_n}$ with periodic boundary condition. Furthermore, we will prove at the end of the proof of this lemma that, for any $n\in\N$,
\begin{equation}\label{gnprop2}
\begin{aligned}
&\alpha_{t_n}^{d(p-1)/p}\|\mu_n\|_{p,R_n\alpha_{t_n}}\\
&\quad\leq \norm g_n^2\norm _{p, R_n}\big[1+C\alpha_{t_n}^{[d(p-1)-2]/2p}\big]+\big[\norm \nabla_{\!\!\scriptscriptstyle{R_n}} g_n\norm _{2, R_n}^2+1\big]C\big[\alpha_{t_n}^{[d(p-1)-2]/2p}+\alpha_{t_n}^{[d(p-1)-2]/p}\big],
\end{aligned}
\end{equation}
where $C$ depends on $d$ and $p$ only. Our assumption $d(p-1)<2$ implies that the exponents at $\alpha_{t_n}$ on the right-hand side are negative. Recalling that $\alpha_{t_n}$ tends to infinity as $t_n\to\infty$ we have (at the cost of choosing a subsequence of $(R_n,t_n,\mu_n)_n$): 
\begin{equation}\label{gnprop3}
\alpha_{t_n}^{d(p-1)/p}\|\mu_n\|_{p,R_n\alpha_{t_n}}\leq \norm g_n^2\norm _{p, R_n}(1+\smfrac 1n)+\smfrac 1n\big(\norm \nabla_{\!\!\scriptscriptstyle{R_n}} g_n\norm _{2, R_n}^2+1\big).
\end{equation}

Note that $g_n$ asymptotically satisfies periodic boundary condition in the box $[-R_n,R_n]^d$. Now we compare it to some version that satisfies zero boundary condition. To this end, pick some $\eps\in(0,1)$ and introduce $\Psi_{R_n}=\bigotimes_{i=1}^d \psi_{R_n}\colon \R^d\to[0,1]$, where $\psi_{R_n}\colon\R\to[0,1]$ is zero outside $[-R_n,R_n]$, one in $[-R_n+R_n^\eps,R_n-R_n^\eps]$ and linearly interpolates between $-R_n$ and $-R_n+R_n^\eps$ and between $R_n-R_n^\eps$ and $R_n$. We are going to estimate the changes of the functionals when going from $g_n$ to $g_n\Psi_{R_n}$. We first use the triangle inequality in \eqref{gndef} and note that the $L^2$-norm of $x\mapsto \alpha_{t_n}^{d/2}\sum_{i=1}^d f_{n,\sigma,i}(\alpha_{t_n}x)$ is not larger than $C\alpha_{t_n}^{-1}\norm \nabla_{\!\!\scriptscriptstyle{R_n}} g_n\norm _{2, R_n}$, where $C\in(0,\infty)$ depends on $d$ only, to see that
$$
\begin{aligned}
\norm g_n\norm _{2, R_n}&\leq \Big|\!\Big|\!\Big| \sqrt{\alpha_{t_n}^d\mu_n\big(\lfloor \alpha_{t_n}\cdot\rfloor\big)}\Big|\!\Big|\!\Big|_{2, R_n} +\Big|\!\Big|\!\Big| \alpha_{t_n}^{d/2}\sum_{i=1}^d f_{n,\sigma,i}(\alpha_{t_n}\,\cdot)\Big|\!\Big|\!\Big| _{2, R_n}\leq 1+C\alpha_{t_n}^{-1}\norm \nabla_{\!\!\scriptscriptstyle{R_n}} g_n\norm _{2, R_n}.
\end{aligned}
$$
Analogously we derive $\norm g_n\norm _{2, R_n}\geq 1-C\alpha_{t_n}^{-1}\norm \nabla_{\!\!\scriptscriptstyle{R_n}} g_n\norm _{2, R_n}$ and thus we get
\begin{equation}\label{gnprop4}
\big|\norm g_n\norm _{2, R_n}-1\big|\leq C\alpha_{t_n}^{-1}\norm \nabla_{\!\!\scriptscriptstyle{R_n}} g_n\norm _{2, R_n},\qquad n\in\N.
\end{equation}
Using Schwarz's inequality and using $n$ so large that $dR_n^{-\eps}\leq 1$, this allows us to show that
$$
\begin{aligned}
\norm \nabla (g_n\Psi_{R_n})\norm _2^2
&\leq\int_{\R^d}\sum_{i=1}^d\big|\partial_i g_n(x)\Psi_{R_n}(x)+g_n(x)\partial_i\Psi_{R_n}(x)\big|^2 \d x \\
&\leq \norm \nabla_{\!\!\scriptscriptstyle{R_n}} g_n\norm ^2_{2, R_n}+2R_n^{-\eps}\norm g_n\norm _{2, R_n}\norm \nabla_{\!\!\scriptscriptstyle{R_n}} g_n\norm _{2, R_n}+dR_n^{-2\eps}\norm g_n\norm _{2, R_n}^2\\
&\leq \norm \nabla_{\!\!\scriptscriptstyle{R_n}} g_n\norm ^2_{2, R_n}+2R_n^{-\eps}\smfrac 12\big(\norm g_n\norm _{2, R_n}^2+\norm \nabla_{\!\!\scriptscriptstyle{R_n}} g_n\norm _{2, R_n}^2\big)+dR_n^{-2\eps}\norm g_n\norm _{2, R_n}^2\\
&\leq \norm \nabla_{\!\!\scriptscriptstyle{R_n}} g_n\norm ^2_{2, R_n}\big[1+R_n^{-\eps}\big]+2R_n^{-\eps}\norm g_n\norm _{2, R_n}^2\\
&\leq \norm \nabla_{\!\!\scriptscriptstyle{R_n}} g_n\norm ^2_{2, R_n}\big[1+R_n^{-\eps}\big]+2R_n^{-\eps}\big(1+2C\alpha_{t_n}^{-1}\norm \nabla_{\!\!\scriptscriptstyle{R_n}} g_n\norm _{2, R_n}+C^2\alpha_{t_n}^{-2}\norm \nabla_{\!\!\scriptscriptstyle{R_n}} g_n\norm _{2, R_n}^2 \big).
\end{aligned}
$$
Hence, we have (probably by choosing again a subsequence of $(R_n,t_n,\mu_n)_n$):
\begin{equation}\label{gnprop6}
\norm \nabla_{\!\!\scriptscriptstyle{R_n}} g_n\norm _{2, R_n}^2\geq \norm \nabla (g_n\Psi_{R_n})\norm _2^2(1-\smfrac 1n)-\smfrac 1n,\qquad n\in\N.
\end{equation}
Furthermore, we note that, without loss of generality, we may assume that
\begin{equation}\label{gnprop7}
\int_{B_{R_n}\setminus B_{R_n-R_n^\eps}} g_n^{2p}(x)\,\d x \leq \frac{|B_{R_n}\setminus B_{R_n-R_n^\eps}|}{|B_{R_n}|}\int_{B_{R_n}} g_n^{2p}(x)\,\d x,\qquad n\in\N.
\end{equation}
This can be easily derived using the shift invariance of the second integral due to periodic boundary conditions. To see this, assume that for every shift $\theta_z(x)=x+z$ modulo $R_n$ with $z\in B_{R_n}$ it holds that:
\begin{equation}\label{shift}
\int_{B_{R_n}\setminus B_{R_n-R_n^\eps}} g_n^{2p}(\theta_z(x))\,\d x > \frac{|B_{R_n}\setminus B_{R_n-R_n^\eps}|}{|B_{R_n}|}\int_{B_{R_n}} g_n^{2p}(x)\,\d x.
\end{equation}
Now, integrate both sides over all $z\in B_{R_n}$, to get a contradiction by changing the order of the integration. Hence, for some $z\in B_{R_n}$, the opposite of \eqref{shift} holds, and we continue to work with $g_n\circ \theta_z$ instead of $g_n$. All properties considered so far are preserved by periodicity.

Note that the quotient on the right-hand side of \eqref{gnprop7} can be estimated against $CR_n^{\eps-1}$ where $C$ does not depend on $n$. Thus, we have:
$$
\norm g_n^2\norm _{p, R_n}^p-\norm (g_n\Psi_{R_n})^2\norm _p^p\leq CR_n^{\eps-1}\norm g_n^2\norm _{p, R_n}^p
$$
which leads (after probably choosing again a subsequence of $(R_n,t_n,\mu_n)_n$) to
\begin{equation}\label{gnprop8}
\norm g_n^2\norm _{p, R_n}\leq \norm (g_n\Psi_{R_n})^2\norm _p(1+\smfrac 1n),\qquad n\in\N.
\end{equation}
Summarizing, substituting \eqref{gnprop1} and \eqref{gnprop3}, and using \eqref{gnprop6} and \eqref{gnprop8}, for any $n$ we have
\begin{equation}\label{mainstep1}
\begin{aligned}
\mbox{R.h.s.~of \eqref{Rntnmun}}
&\leq \norm g_n^2\norm _{p, R_n}(1+\smfrac 1n)-\norm \nabla_{\!\!\scriptscriptstyle{R_n}} g_n\norm _{2, R_n}^2(1-\smfrac 2n)-\smfrac 1n \norm \nabla_{\!\!\scriptscriptstyle{R_n}} g_n\norm _{2, R_n}^2+\smfrac 2n\\
&\leq \norm (g_n\Psi_{R_n})^2\norm _p(1+\smfrac 3n)-\norm \nabla(g_n\Psi_{R_n})\norm _2^2(1-\smfrac 3n)-\smfrac 1n \norm \nabla_{\!\!\scriptscriptstyle{R_n}} g_n\norm _{2, R_n}^2+\smfrac 3n\\
&=\Big[\frac{\norm (g_n\Psi_{R_n})^2\norm _p}{\norm g_n\Psi_{R_n}\norm _2^2}\,\frac{1+\smfrac 3n}{1-\smfrac 3n}-\frac{\norm \nabla(g_n\Psi_{R_n})\norm _2^2}{\norm g_n\Psi_{R_n}\norm _2^2}\Big](1-\smfrac 3n)\norm g_n\Psi_{R_n}\norm _2^2-\smfrac 1n \norm \nabla_{\!\!\scriptscriptstyle{R_n}} g_n\norm _{2, R_n}^2+\smfrac 3n.
\end{aligned}
\end{equation}
Now observe that $h_n=g_n\Psi_{R_n}/\norm g_n\Psi_{R_n}\norm _2$ is an $L^2$-normalized element of $H^1(\R^d)$ and of $L^{2p}(\R^d)$. Hence, we may estimate the term in the brackets against the supremum over all such functions, which is equal to $\rho_{d,p}^{\ssup {\rm c}}(\frac{1+3/n}{1-3/n})$, see \eqref{rhocdef}. Since $\rho_{d,p}^{\ssup {\rm c}}(\frac{1+3/n}{1-3/n})>0$ by \eqref{rhoident} and, obviously, $\norm g_n\Psi_{R_n}\norm _2^2\leq \norm g_n\norm _{2, R_n}^2$, we can proceed with
$$
\begin{aligned}
\mbox{R.h.s.~of \eqref{Rntnmun}}
&\leq \rho_{d,p}^{\ssup {\rm c}}\big(\smfrac{1+3/n}{1-3/n}\big)(1-\smfrac 3n)\norm g_n\norm _{2, R_n}^2-\smfrac 1n \norm \nabla_{\!\!\scriptscriptstyle{R_n}} g_n\norm _{2, R_n}^2+\smfrac 3n\\
&\leq \rho_{d,p}^{\ssup {\rm c}}\big(\smfrac{1+3/n}{1-3/n}\big)\norm g_n\norm _{2, R_n}^2-\smfrac 1n \norm \nabla_{\!\!\scriptscriptstyle{R_n}} g_n\norm _{2, R_n}^2+\smfrac 3n.
\end{aligned}
$$
By \eqref{gnprop4} and at the cost of chosing again a subsequence of $(R_n,t_n,\mu_n)_n$, we have that $\norm g_n\norm _{2, R_n}^2\leq 1+\smfrac 1n\norm \nabla_{\!\!\scriptscriptstyle{R_n}} g_n\norm _{2, R_n}^2/\rho_{d,p}^{\ssup {\rm c}}(\smfrac{1+3/n}{1-3/n})$. Using this in the last display, we arrive for all $n$, at
$$
\mbox{R.h.s.~of \eqref{Rntnmun}}\leq \rho_{d,p}^{\ssup {\rm c}}\big(\smfrac{1+3/n}{1-3/n}\big)+\smfrac 3n.
$$
Recalling \eqref{rhoident}, we see that the right-hand side converges to $\rho_{d,p}^{\ssup {\rm c}}(1)$ as $n\uparrow\infty$. This ends the proof of the lemma.

Now we give the proof of \eqref{gnprop2}. Recall \eqref{fndef} and \eqref{gndef} and that we write $f_{n,\sigma,i}(y)$ instead of $f_{n,\sigma(y),i}(y)$. The triangle inequality gives that
\begin{equation}\label{2proof1}
\begin{aligned}
\norm g_n\norm _{2p, R_n}&\geq \alpha_{t_n}^{d/2}\norm \mu_n(\lfloor \alpha_{t_n} \,\cdot\,\rfloor)^{1/2} \norm _{2p}-\alpha_{t_n}^{d/2}\Big|\!\Big|\!\Big| \sum_{i=1}^d f_{n,\sigma,i}(\alpha_{t_n}\,\cdot\, )\Big|\!\Big|\!\Big| _{2p, R_n}\\
&=\alpha_{t_n}^{d(p-1)/2p}\|\mu_n\|_p^{\frac 12}
-\alpha_{t_n}^{d(p-1)/2p}\Big|\!\Big|\!\Big| \sum_{i=1}^d f_{n,\sigma,i}\Big|\!\Big|\!\Big| _{2p, R_n}.
\end{aligned}
\end{equation}
Now we estimate, for any $y\in\R^d$,
\begin{equation}\label{pnorm2norm}
\begin{aligned}
\Big|\sum_{i=1}^d f_{n,\sigma,i}(y)\Big|^{2p}
&\leq d^{2p}\sum_{i=1}^d \Big|\sqrt{\mu_n\big(\lfloor y \rfloor+{\e}_{\sigma(1)}+\dots+{\e}_{\sigma(i)}\big)}-\sqrt{\mu_n\big(\lfloor y \rfloor+{\e}_{\sigma(1)}+\dots+{\e}_{\sigma(i-1)}\big)}\Big|^{2p}\\
&\leq d^{2p}\sum_{i=1}^d\Big|\sqrt{\mu_n\big(\lfloor y \rfloor+{\e}_{\sigma(1)}+\dots+{\e}_{\sigma(i)}\big)}-\sqrt{\mu_n\big(\lfloor y \rfloor+{\e}_{\sigma(1)}+\dots+{\e}_{\sigma(i-1)}\big)}\Big|^{2}\\
&=d^{2p}\alpha_{t_n}^{-d-2}|\nabla_{\!\!{\scriptscriptstyle{R_n}}} g_n(y/\alpha_{t_n})|^2,
\end{aligned}
\end{equation}
since the term in brackets is not larger than one (recall that $\mu_n$ is a probability measure on a finite set). Using this in \eqref{2proof1}, we obtain
$$
\alpha_{t_n}^{d(p-1)/2p}\|\mu_n\|_p^{\frac 12}\leq \norm g_n\norm _{2p, R_n}+d\alpha_{t_n}^{[d(p-1)-2]/2p}\norm \nabla_{\!\!\scriptscriptstyle{R_n}} g_n\norm _{2, R_n}^{\frac 1p}.
$$
Now square both sides and use the estimates 
$$
\norm \nabla_{\!\!\scriptscriptstyle{R_n}} g_n\norm _{2, R_n}^{\frac 1p}\leq 1+\norm \nabla_{\!\!\scriptscriptstyle{R_n}} g_n\norm _{2, R_n} \qquad\mbox{and}\qquad \norm g_n\norm _{2p, R_n}\norm \nabla_{\!\!\scriptscriptstyle{R_n}} g_n\norm _{2, R_n}\leq \frac 12\big(\norm g_n\norm _{2p, R_n}^2+\norm \nabla_{\!\!\scriptscriptstyle{R_n}} g_n\norm _{2, R_n}^2\big)
$$
and summarize to arrive at \eqref{gnprop2}.

\end{proof}

\subsection*{Acknowledgement.} We are thankful to an anonymous referee who carefully read the earlier draft and gave many helpful comments.

\end{document}